\documentclass{amsart}
\usepackage{amsmath}%
\usepackage{amsfonts}%
\usepackage{amssymb}%
\usepackage{graphicx}
\usepackage{latexsym}
\newcommand{\R}{I\kern-0.37emR}
\newcommand{\boldgreek}[1]{\mbox{\boldmath$#1$}}

\newcommand{\ny}{n\rightarrow\infty}
\newcommand{\Q}{I\kern-0.37emP}
\newcommand{\E}{I\kern-0.37emE}
\def\bgk{\boldgreek}
\def\bbe{\bgk \beta}
\newtheorem{theorem}{Theorem}
\theoremstyle{plain}

\newtheorem{remark}{Remark}

\numberwithin{equation}{section}
\begin{document}
\title[Empirical regression quantile process]{Empirical regression quantile process with possible application to risk analysis}
\author{Jana Jure\v{c}kov\'a}
\address[J. Jure\v{c}kov\'a]
{Department of Probability and Statistics, Charles University, \newline%
\indent Faculty of Mathematics and Physics, Prague, Czech Republic}%
\email{jurecko@karlin.mff.cuni.cz}%
\urladdr{http://www.karlin.mff.cuni.cz/~jurecko}
\author{Martin Schindler}
\address[M. Schindler]{Department of Applied Mathematics, Technical University\newline%
\indent Liberec, Czech Republic}%
\email{martin.schindler@tul.cz}%
\author{Jan Picek}
\address[J. Picek]{Department of Applied Mathematics, Technical University\newline%
\indent Liberec, Czech Republic}
\email{jan.picek@tul.cz}%
\thanks{The authors gratefully acknowledge the support of the Grant GA\v{C}R 15-00243S}
\date{}

\begin{abstract}
The processes of the averaged regression quantiles and of their modifications provide useful tools in the regression models when the covariates are not fully under our control. As an application we mention the probabilistic risk assessment in the situation when the  return depends on some exogenous variables. 
The  processes enable to evaluate the expected $\alpha$-shortfall ($0\leq\alpha\leq 1$) and other measures of the risk, recently generally accepted in the financial literature, but also help to measure the risk in environment ana\-ly\-sis and elsewhere. 
\end{abstract}
\maketitle

\section{Introduction}

\noindent 
In everyday life and  practice  we encounter various risks,  depending on various contributors. The risk contributors may be partially under our control, and information on them is important, because it helps to make good decisions about system design. This problem appears not only in the financial market, insurance and social statistics, but also in environment ana\-ly\-sis dealing with exposures to toxic chemicals (coming from power plants, road vehicles,
agriculture), and elsewhere; see  \cite{Molak} for an excellent  review of  such problems. 
Our aim is to analyze the risks with the aid of probabilistic risk assessment.
In the literature were recently defined various  coherent risk measures, some 
satisfying suitable axioms.     We refer to \cite{Artzner1997}, \cite{Artzner1999},  \cite{Pflug2000}, \cite{Tasche}, \cite{Uryasev2000}, {\cite{RockUryasev2001}, \cite{AcerbiTasche2002}, \cite{Trindade}, \cite{CaiWang}, \cite{RockafellarUryasev2013}, \cite{RockMiranda}, and to other papers cited in,  for discussions and some projects. For possible applications in the insurance we refer to \cite{Chan}. 

A generally accepted  measure of the risk is the expected shortfall, based on quantiles of a portfolio return.  Its properties were recently intensively studied. Acerbi and Tasche in \cite{AcerbiTasche2002} speak on "expected loss in the 100$\alpha$\% worst cases", or shortly on "expected $\alpha$-shortfall", $0<\alpha<1,$ which is defined as
\begin{equation}\label{shortfall}
-\E\{Y|Y\leq F^{-1}(\alpha)\}=-\frac{1}{\alpha}\int_0^{\alpha}F^{-1}(u)du,
\end{equation}
where $F$ is the distribution function of the asset $Y.$ The quantity can be estimated by means of approximations of the quantile function $F^{-1}(u)$ by  the sample quantiles. 

The quantile regression is an important method for investigation of the risk of an assett in the situation that it depends on some exogenous variables. An averaged regression quantile, introduced in \cite{JP2014}, or some of its modifications, serve as a convenient tool for the global risk measurement in such a situation, when the amount of covariates is not under our control. 
The typical model for the relation of the loss to the covariates is  the regression model
\begin{equation}
\label{1}
Y_{ni}=\beta_0+{\bf x}_{ni}^{\top}{\boldgreek\beta}+e_{ni}, \quad
i=1,\ldots,n
\end{equation}
where $Y_{n1},\ldots,Y_{nn}$ are observed responses, $e_{n1},\ldots,e_{nn}$ are 
independent model errors,  possibly 
non-identically distributed with unknown distribution functions $F_i, \; i=1,\ldots,n.$
The covariates ${\bf x}_{ni}=(x_{i1},\ldots,x_{ip})^{\top}, \; i=1,\ldots,n$ are random or nonrandom,
and
${\boldgreek\beta}^*=(\beta_0,\boldgreek\beta^{\top})^{\top}=(\beta_0,\beta_1,\ldots,\beta_p)^{\top}\in\R^{p+1}$ 
is an unknown parameter. For the sake of brevity, we also use the notation
 $\mathbf x^*_{ni}=(1,x_{i1},\ldots,x_{ip})^{\top}, \; i=1,\ldots,n.$
  
An important tool in the risk analysis is the regression $\alpha$-quantile  $$\widehat{\boldgreek\beta}^*_n(\alpha)=\left(\hat{\beta}_{n0}(\alpha),(\widehat{\boldgreek\beta}_n(\alpha))^{\top}\right)^{\top}=
\left(\hat{\beta}_{n0}(\alpha),\hat{\beta}_{n1}(\alpha),\ldots,\hat{\beta}_{np}(\alpha)\right)^{\top}.$$
It is a $(p+1)$-dimensional vector defined as a minimizer
\begin{eqnarray}\label{RQ}
&&\widehat{\boldgreek\beta}^*_n(\alpha)=\arg\min_{\mathbf b\in{\R}^{p+1}}\Big\{\sum_{i=1}^n\Big[\alpha(Y_i-\mathbf x_i^{*\top}\mathbf b)^++(1-\alpha)(Y_i-\mathbf x_i^{*\top}\mathbf b)^-\Big]\Big\}\nonumber\\
&& \; \mbox{where } \; z^+=\max(z,0) \; \mbox{ and} \; z^-=\max(-z,0), \; z\in\R_1. 
\end{eqnarray} 
The solution $\widehat{\bbe}_n^*(\alpha)=(\hat{\beta}_0(\alpha), \widehat{\bbe}(\alpha))^{\top}$ minimizes the $(\alpha, 1-\alpha)$ convex combination of residuals $(Y_i-\mathbf x_i^{*\top}\mathbf b)$ over $\mathbf b\in \R^{p+1},$ where the choice of $\alpha$ depends on the balance between underestimating and overestimating the respective losses $Y_i.$ The increasing $\alpha\nearrow 1$  reflects a greater concern about underestimating losses Y, comparing to overestimating.

The methodology is based on the \textit{averaged regression} $\alpha$-\textit{quantile}, what is the following
 weighted mean of components of $\widehat{\boldgreek\beta}_n^*(\alpha), \; 0\leq\alpha\leq 1$: 
\begin{equation}\label{22x}
\bar{B}_n(\alpha)=\overline{\mathbf x}_n^{*\top}\widehat{\boldgreek\beta}_n^*(\alpha)=\widehat{\beta}_{n0}(\alpha)+\frac 1n\sum_{i=1}^n\sum_{j=1}^p x_{ij}\widehat{\beta}_j(\alpha), \quad \overline{\mathbf x}_n^{*}=\frac 1n\sum_{i=1}^n\mathbf x_i^* 
\end{equation}
	In \cite{JP2014} it was shown that $\bar{B}_n(\alpha)-\beta_0-\bar{\mathbf x}_n^{\top}\bbe$ is asymptotically equivalent to the $[n\alpha]$-quantile $e_{n:[n\alpha]}$ of the model errors, if they are identically distributed. Hence, $\bar{B}_n(\cdot)$ can help to make an inference on the expected $\alpha$-shortfall (\ref{shortfall}) even under the nuisance regression.

Besides $\bar{B}_n(\alpha)$, its various modifications can also be used, sometimes better comprehensible. The methods are nonparametric, thus applicable also to heavy-tailed  and skewed distribution; notice that \cite{Braione} speak about considerable improvement over normality, trying to use different distributions. An extension to autoregressive models is possible and will be a subject of the further study; there the main tool will be the autoregression quantiles, introduced in \cite{KoulSaleh1995}, and their averaged versions. The autoregression quantile will reflect the value-at-risk, based on the past assets, while its averaged version will try to mask the past history.

The behavior of $\bar{B}_n(\alpha)$ with $0<\alpha<1$ has been illustrated in \cite{Bassett1988} and \cite{KB1982}, and summarized in \cite{Koenker2005}; here it is showed that
$\bar{B}_n(\alpha)$ is nondecreasing step function of $\alpha\in(0,1).$ 
The extreme $\bar{B}_n(1)$ with $\alpha=1$ was studied in \cite{Jur2016}.
  Notice that the upper bound of the number $J_n$ of breakpoints of $\widehat{\bbe}_n^*(\cdot)$ and also of $\bar{B}_n(\cdot)$ is  $\left(\begin{array}{c} n\\ p+1\\ \end{array}\right)=\mathcal O\left(n^{p+1}\right).$ However, Portnoy in \cite{Portnoy1991} showed that, under some condition on the design matrix $\mathbf X_n,$ the number $J_n$ of breakpoints is of order $\mathcal O_p(n \; \log n)$ as $\ny,$ and thus much smaller. 
  
   An alternative to the regression quantile is the  \textit{two-step regression $\alpha$-quantile}, introduced in \cite{JP2005}. Here the slope components $\bbe$  are estimated  by a specific rank-estimate $\tilde{\bbe}_{nR},$ which is invariant to the shift in location. The intercept component is then estimated by the $\alpha$-quantile of residuals of $Y_i$'s from $\tilde{\bbe}_{nR}.$ The averaged two-step regression quantile $\tilde{B}_n(\alpha)$ is asymptotically equivalent to  $\bar{B}_n(\alpha)$ under a wide choice of the R-estimators of the slopes. However,  finite-sample behavior of $\tilde{B}_n(\alpha)$ generally differs from that of $\bar{B}_n(\alpha);$ is affected by the choice of R-estimator, but the main difference is that the number of breakpoints of $\tilde{B}_n(\alpha)$ exactly equals to $n$.   
   
Being aware of various important applications of the problem, we shall study this situation in more detail. The averaged regression quantile $\bar{B}_n(\alpha)$ is monotone in $\alpha,$ while the two-step averaged regression quantile $\tilde{B}_n(\alpha)$ can be made motonone by a suitable choice of R-estimate $\tilde{\bbe}_{nR}.$  Hence, we can consider their inversions, which in turn estimate the parent distribution $F$ of the model errors. As such they both provide a tool for an inference. The behavior of these processes and of their approximations is analyzed and numerically illustrated. 
\section{Behavior of $\bar{B}_n(\alpha)$ over $\alpha\in(0,1).$} 
Let us  first describe one possible form of the averaged regression quantile $\bar{B}_n(\alpha)$ as a weighted mean of the basic components of vector $\mathbf Y.$ 
Consider again the minimization (\ref{RQ}), fixed $\alpha\in [0,1]$ fixed. This was treated in \cite{KB1978} as a special linear programming problem, and later on various modifications of this algorithm were developed. Its dual program 
is a parametric linear program, which can be written simply as
\begin{eqnarray}
\label{linprog}
&\mbox{ maximize \ } & {\bf Y}_{n}^{\top}\hat{{\bf a}}(\alpha)
\nonumber\\
&\mbox{under} & \mathbf X_n^{*\top}\hat{\mathbf a}(\alpha)=(1-\alpha)\mathbf X_n^{*\top}\mathbf 1_n^{\top}\\             
&& \hat{{\mathbf a}}(\alpha) \in [0,1]^n, \ 0\leq\alpha\leq 1\nonumber
\end{eqnarray}
where
\begin{equation}\label{matrix}
 \mathbf X_n^*=\left[\begin{array}{l}  \mathbf x_{n1}^{*\top}\\  \ldots\\  \mathbf x_{nn}^{*\top}\end{array}\right] \quad \mbox{ is of order } n\times(p+1).
\end{equation} 
The components of the optimal solution $\hat{\bf a}(\alpha)=(\hat{a}_{n1}(\alpha),\ldots,\hat{a}_{nn}(\alpha))^{\top}$ of
(\ref{linprog}), called 
regression rank scores, were studied in\cite{GJ92}, who showed that $\hat{a}_{ni}(\alpha)$ is a continuous, piecewise linear
function of $\alpha\in[0,1]$ and $\hat{a}_{ni}(0)=1, \; \hat{a}_{ni}(1)=0, \; i=1,\ldots,n.$ Moreover, $\hat{\bf a}(\alpha)$ is invariant in the sense that it does not change if $\mathbf Y$ is replaced with $\mathbf Y+\mathbf X_n^*\mathbf b^*, \; \forall \mathbf b^*\in\R^{p+1}$ (see \cite{GJ92} for detail).

Let $\{\mathbf x_{i_1}^*,\ldots,\mathbf x_{i_{p+1}}^*\}$ be the optimal base in (\ref{linprog}) and let $\{Y_{i_1},\ldots, Y_{i_{p+1}}\}$ be the corresponding responses in model (\ref{1}).
Then $\bar{B}_n(\alpha)$  
equals to a weighted mean of $\{Y_{i_1},\ldots, Y_{i_{p+1}}\}$, with the weights based on the regressors. Indeed, we have a theorem
\begin{theorem}
\label{Theorem1}
Assume that the regression matrix (\ref{matrix})
has full rank $p+1$ and that the distribution functions $F_1,\ldots,F_n$ of model errors are continuous and increasing in $(-\infty,\infty).$ Then with probability 1
\begin{eqnarray}\label{main0}
&&\bar{B}_n(\alpha)=\sum_{k=1}^{p+1} w_{k,\alpha}Y_{i_k}, \quad 
\sum_{k=1}^{p+1} w_{k,\alpha}=1\\[2mm]
and \quad &&\bar{B}_n(\alpha)\leq \bar{B}_n(1)<\max_{i\leq n}Y_i\label{main01}
\end{eqnarray}
where 
the vector $\mathbf Y_n(1)=(Y_{i_1},\ldots,Y_{i_{p+1}})^{\top}$ 
corresponds to the optimal base of the linear program (\ref{linprog}). 

The vector $\mathbf w_{\alpha}=(w_{1,\alpha},\ldots,w_{p+1,\alpha})^{\top}$ of coefficients equals to
\begin{equation}\label{w}
\mathbf w_{\alpha}=\left[n^{-1}\mathbf 1_n^{\top}\mathbf X_n^*(\mathbf X_{n1}^*)^{-1}\right]^{\top}, \; \mbox{ while } \;  \sum_{k=1}^{p+1}w_{k,\alpha}=1
\end{equation}
where $\mathbf X_{n1}^*$ is the submatrix of $\mathbf X_n^*$ with the rows $\mathbf x_{i_1}^{*\top},\ldots,\mathbf x_{i_{p+1}}^{*\top}.$
\end{theorem}
\begin{proof}
The regression quantile $\widehat{\bbe}_n^*(\alpha)$ is a step function of $\alpha\in(0,1).$ If $\alpha$ is a
continuity point of the regression quantile trajectory, then we have the following identity,  proven in \cite{JP2014}:
\begin{equation}
\label{tail5}
 \bar{B}_n(\alpha)=\frac{1}{n}\sum_{i=1}^n{\bf x}_i^{*\top}\widehat{\boldgreek\beta}_n^*(\alpha)=
-\frac{1}{n}\sum_{i=1}^n Y_i\hat{a}_{ni}^{\prime}(\alpha) 
\end{equation} 
where $\hat{a}_{ni}^{\prime}(\alpha))=\frac{d}{d\alpha}\hat{a}_{ni}(\alpha).$ 
Moreover, (\ref{linprog}) implies
\begin{eqnarray}
\label{tail2}
&&\sum_{i=1}^n\hat{a}_{ni}^{\prime}(\alpha)= -n\\
&&\sum_{i=1}^nx_{ij}\hat{a}_{ni}^{\prime}(\alpha) =-\sum_{i=1}^n x_{ij}, \qquad  1\leq j \leq p,\nonumber 
\end{eqnarray}
Notice that 
$\hat{a}_{ni}^{\prime}(\alpha)\neq 0$  iff $\alpha$ is the point of continuity of $\widehat{\bbe}_n^*(\cdot)$ and $Y_i=\mathbf x_i^{*\top}\widehat{\boldgreek\beta}_n^*(\alpha).$ 
To every fixed continuity point $\alpha$ correspond exactly $p+1$ such components, such that the corresponding   
$\mathbf x_i^*$ belongs to the optimal base of program (\ref{linprog}).  
Hence there exist coefficients $w_{k,\alpha}, \; k=1,\ldots,p+1$ such that 
$$\bar{B}_n(\alpha)=-\frac 1n\sum_{i=1}^n Y_i\hat{a}_{ni}^{\prime}(\alpha)=
\sum_{k=1}^{p+1}w_{k,\alpha}Y_{i_k}.$$

The  
equalities $Y_i=\mathbf x_i^{*\top}\widehat{\boldgreek\beta}_n^*(\alpha)$ hold  just for $p+1$ components of the optimal base $\mathbf x_{i_1}^*,\ldots,\mathbf x_{i_{p+1}}^*$.   
Let $\mathbf X_{n1}^*$ be the submatrix of $\mathbf X_n^*$ with the rows $\mathbf x_{i_1}^{*\top},\ldots,\mathbf x_{i_{p+1}}^{*\top}$ and let $(\hat{\mathbf a}_1^{\prime}(\alpha))^{\top}=(\hat{a}_{i_1}^{\prime}(\alpha),\ldots,\hat{a}_{i_{p+1}}^{\prime}(\alpha)).$ Then $\mathbf X_{n1}^*$ is regular with probability 1 and
$$\mathbf w_{\alpha}^{\top}=-\frac 1n(\hat{\mathbf a}^{\prime}(\alpha))^{\top}= \frac 1n\mathbf 1_n^{\top}\mathbf X_n^*(\mathbf X_{n1}^*)^{-1}.$$
$$(\hat{\mathbf a}_1^{\prime}(\alpha))^{\top}=-\mathbf 1_n^{\top}\mathbf X_n^*(\mathbf X_{n1}^*)^{-1}  \; \mbox{ and } \; \sum_{k=1}^{p+1}w_{k,\alpha}=1.$$
This and (\ref{tail5}) imply (\ref{main0}) and (\ref{w}). The inequality (\ref{main01}) was proven in \cite{Jur2016}.
\end{proof}

Let us now consider $\bar{B}_n(\alpha)$ as a process in $\alpha\in(0,1).$ Assume that  
all model errors $e_{ni}, \; i=1,\ldots,n$ are independent and equally distributed according to joint continuous increasing distribution function $F.$ We are interested in the \textit{the average regression quantile process} 
$$\bar{\mathcal B}_n(\alpha)
=\Big\{n^{1/2}\bar{\mathbf x}_n^{*\top}\Big(\widehat{\boldgreek\beta}_n^*(\alpha)-\check{\boldgreek\beta}(\alpha)\Big); \; 0<\alpha<1\Big\}$$
where $\check{\boldgreek\beta}(\alpha)=(F^{-1}(\alpha)+\beta_0,\beta_1,\ldots,\beta_p)^{\top}$ is the population counterpart of the regression quantile. 
As proven in  \cite{Jur2017}, 
 the process $\bar{\mathcal B}_n$ converges to a Gaussian process in the Skorokhod topology as $\ny,$  under  mild conditions on $F$ and $\mathbf X_n.$ More precisely, 
\begin{equation}\label{5.5b}
\bar{\mathcal B}_n\stackrel{\mathcal D}{\rightarrow}(f(F^{-1}))^{-1}W^* \; \mbox{ as } \; \ny
\end{equation}
where $W^*$ is the Brownian bridge on (0,1). 

However, we are rather interested in behavior of the process $\bar{\mathcal B}_n$ under a finite number of observations. The trajectories of $\bar{\mathcal B}_n$ are step functions, nondecreasing in $\alpha\in(0,1),$ and  they have finite numbers of discontinuities for each $n$.   As shown in \cite{BK82}, if 
$\bar{B}_n(\alpha_1)= \bar{B}_n(\alpha_2)$ for $0<\alpha_1<\alpha_2<1,$ then $\alpha_2-\alpha_1\leq \frac{p+1}{n}$ with probability 1, then the length of interval, on which is $\bar{B}_n(\alpha)$ constant, tends to 0 for $\ny$ and fixed $p.$ Let $0<\alpha_1<\ldots<\alpha_{J_n}<1$ be the breakpoints of $\bar{B}_n(\alpha), \; 0< \alpha<1,$ and $-\infty<Z_1<\ldots<Z_{J_n+1}<\infty$ be the corresponding values of $\bar{B}_n(\alpha)$ between the breakpoints. Then we can consider the inversion $\hat{F}_n(z)$ of $\bar{B}_n(\alpha),$  namely
$$\hat{F}_n(z)=\inf\{\alpha: \; \bar{B}_n(\alpha)\geq z\}, \; -\infty<z<\infty.$$
It is a bounded nondecreasing step function and, given $Y_1,\ldots,Y_n$ satisfying (\ref{1}), $\hat{F}_n$ is a distribution function of a random variable attaining values $Z_1,\ldots,Z_{J_n+1}$ with probabilities equal to the spacings of $0,\alpha_1,\ldots,\alpha_{J_n},1.$ The tightness of the empirical process $\hat{F}_n$ and its convergence to $F$ was studied in \cite{Portnoy1984} under some specific conditions;  $\hat{F}_n$ is recommended as an estimate of $F,$ which  would enable e.g. goodness-of-fit testing about $F$ in the presence of a nuisance regression.

\section{Properties of the averaged two-step regression quantile $\widetilde{\mathcal B}_n(\alpha)$}

While the advantage of $\bar{B}_n(\alpha)$ is in its monotonicity, the inference based on the process $\widetilde{\mathcal B}_n(\alpha)$ can be more comprehensible.  Hence, we can consider the empirical process $\widetilde{\mathcal B}_n(\alpha)$ based on two-step regression quantiles $\widetilde{\bbe}_n(\alpha)$ as an alternative to $\widehat{\bbe}_n(\alpha).$ Both processes are asymptotically equivalent as $\ny.$  

The two-step regression $\alpha$-quantile treats the slope components $\bbe$ and the intercept $\beta_0$ separately.
The slope component part  
is an R-estimate $\widetilde{\bbe}_{nR}$ of  
$\bbe.$ Its advantage is that it is invariant to the shift in location, hence independent of $\beta_0.$ It starts with
selection of a nondecreasing function $\varphi(u), \; u\in(0,1)$, square-integrable on (0,1).
   Then we can consider two types of rank scores, generated by $\varphi:$
\begin{enumerate}
	\item \begin{equation} \label{exact }
	\mbox{Exact scores:} \quad \mathcal A_n(i)=\E\{\varphi(U_{n:i})\},   \;  i=1,\ldots,n
	\end{equation}
\smallskip	
where $U_{n:1}\leq \ldots \leq U_{n:n}$ is the ordered random sample of size $n$ from the uniform (0,1) distribution. 
\item \begin{equation}\label{approx}
\mbox{Approximate scores:}   \begin{array}{rll} 
\quad either & (i) &  A_n(i)=n~\int_{(i-1)/n}^{i/n}\varphi(u)du,\\
                               or & (ii) & A_n(i)=\varphi\left( \frac{i}{n+1}\right), \; i=1,\ldots,n.\\
\end{array}	
\end{equation}																																								

\end{enumerate}
The test criteria and estimates based on either of these scores are asymptotically equivalent as $\ny;$ but the rank tests based on the exact scores are locally most powerful against pertinent alternatives under finite $n$.
The R-estimator $\widetilde{\bbe}_{nR}$ of the slopes is a minimizer of the \cite{Jaeckel1972} measure of rank dispersion ${\mathcal D}_n({\mathbf b}):$
\begin{eqnarray}
\label{11}
&&\widetilde{\boldgreek\beta}_{nR}={\rm argmin}_{{\bf b}\in\mathbb{R}^p}{\mathcal D}_n({\mathbf b}),\\
\mbox{where }\quad && 
{\mathcal D}_n({\bf b})=\sum_{i=1}^n(Y_i-{\bf x}_i^{\top}{\bf b})~\mathcal A_n(R_{ni}(Y_i-\mathbf x_i^{\top}{\bf b})), \quad \mathbf b\in\mathbb R^p\nonumber
\end{eqnarray}
and where $\mathcal A_n(\cdot)$ can be replaced with $A_n(\cdot).$ Here
$R_{ni}(Y_i-\mathbf x_i^{\top}{\bf b})$ is the rank of the $i$-th residual, $i=1,\ldots,n.$  
The intercept component $\tilde{\beta}_{n0}(\alpha)$ pertaining to the two-step regression $\alpha$-quantile is defined as the $[n\alpha]$-order statistic of the residuals $Y_i-\mathbf x_i^{\top}\widetilde{\boldgreek\beta}_{nR}, \; i=1,\ldots,n.$ The two-step $\alpha$-regression quantile is then the vector
\begin{equation}\label{11a}
\widetilde{\bbe}^*_n(\alpha)=\left(\begin{array}{c}
\tilde{\beta}_{n0}(\alpha)\\
\widetilde{\boldgreek\beta}_{nR}\\
\end{array}
\right)\in\mathbb R^{p+1}.
\end{equation}
The typical choice of $\varphi$ is the following:
\begin{equation}
\label{8}
\varphi_{\lambda}(u)=\lambda-I[u<\lambda],
 \quad 0<u<1, \quad 0<\lambda<1
\end{equation} 
combined with  the approximate scores (ii) in (\ref{approx}). 
These scores were originated in \cite{Hajek65}; he used the following scores [now known as H\'ajek's rank scores]: 
\begin{equation} \label{Hajek}
a_i(\lambda,\mathbf b)=\left\{\begin{array}{lll}
0  &\ldots&\quad R_{ni}(Y_i)<n\lambda\\[1mm]
R_{ni}(Y_i)-n\lambda \quad &\ldots&\quad n\lambda\leq R_{ni}(Y_i)<n\lambda+1\\[1mm]
1                 &\ldots&\quad n\lambda+1\leq R_{ni}(Y_i).\\
\end{array}\right.
\end{equation}
The solutions of  (\ref{11}) are generally not uniquely
determined. We can e.g.
take the center of gravity of the set of all solutions; 
however, the asymptotic representations and
distributions apply to any solution.

Define the \textit{averaged two-step regression $\alpha$-quantile} $\widetilde{B}_n(\alpha)$   as 
\begin{equation}\label{21a}
\widetilde{B}_n(\alpha)=\bar{\mathbf x}_n^{*\top}\widetilde{\bgk\beta}_n^*(\alpha).
\end{equation}
By (\ref{11}),
\begin{equation}\label{11x} 
\widetilde{B}_n(\alpha)=\left(Y_i-(\mathbf x_i-\bar{\mathbf x}_n)^{\top}\widetilde{\bbe}_{nR}\right)_{n:[n\alpha]},
\end{equation}
hence it is equal to the $[n\alpha]$-th order statistic of the residuals $Y_i-(\mathbf x_i-\bar{\mathbf x}_n)^{\top}\widetilde{\bbe}_{nR}, \; i=1,\ldots,n.$
Then $\widetilde{B}_n(\alpha)$ is obviously
scale equivariant and regression equivariant. 
\cite{JP2005} originally considered the two-step regression $\alpha$-quantile with $\lambda=\alpha$ in (\ref{8}) for each $\alpha\in(0,1)$ in the R-estimator of the slopes.  
 The  averaged two-step version corresponding to this choice is very close to $\bar{B}_n(\alpha),$ but for finite $n$ it is generally not monotone in $\alpha.$  However, it suffices to consider $\lambda\in(0,1)$ fixed, independent of $\alpha;$  this makes $\widetilde{B}_n(\alpha)$ monotone in $\alpha,$ thus invertible and simpler. 
$\widetilde{B}_n(\alpha)$   is asymptotic equivalent
   to $\bar{B}_n(\alpha)$  under general conditions, hence also
   asymptotically equivalent to $e_{n:[n\alpha]}+\beta_0+\bar{\mathbf x}_n^{\top}\bbe.$ Hence $\widetilde{B}_n(\alpha)$ is a convenient tool for an inference under a nuisance regression. 
The asymptotic equivalence of $\widetilde{B}_n(\alpha)$  and $\bar{B}_n(\alpha)$ will be proven under the following mild conditions on $F$ and on $\mathbf X_n=\left[\mathbf x_{n1},\ldots,\mathbf x_{nn}\right]^{\top}:$
\begin{description}
	\item[{(A1)}] \textit{Smoothness of $F$:} The errors $e_{ni}, \; i=1,\ldots,n$ are independent and identically distributed. Their distribution function $F$ has an absolutely continuous density 
 and positive and finite Fisher's information:
	$$0<\mathcal I(f)=\int \left(\frac{f^{\prime}(z)}{f(z)}\right)^2dF<\infty.$$
 \item[{(A2)}] \textit{Noether's condition on regressors:}
\begin{eqnarray*}
&&\lim_{\ny}\mathbf Q_n=\mathbf Q \quad\mbox{ where } \mathbf Q_n=n^{-1}\sum_{i=1}^n(\mathbf x_i-\bar{\mathbf x}_n)(\mathbf x_i-\bar{\mathbf x}_n)^{\top}\\ 
&& \mbox{ and } \mathbf Q \; \mbox{ is positively definite } \; p\times p \; \mbox{ matrix; moreover,}\\ 
&&\lim_{\ny}\max_{1\leq i\leq n}~n^{-1}(\mathbf x_i-\bar{\mathbf x}_n)^{\top}\mathbf Q_n^{-1}(\mathbf x_i-\bar{\mathbf x}_n)=0. 
\end{eqnarray*} 
 \item[{(A3)}] \textit{Rate of regressors:}
 $$ \max_{1\leq i\leq n}\|\mathbf x_{ni}-\bar{\mathbf x}_n\|=o(n^{1/4}) \; \mbox{ as } \; \ny.$$  	
\end{description}
We shall prove the asymptotic equivalence for R-estimators based on score function $\varphi_{\lambda}, \; 0<\lambda<1,$ because of its simplicity. 
However, an analogous proof applies to an R-estimator generated by any nondecreasing and square-integrable function $\varphi.$
\begin{theorem}\label{TheoremA}
Let $\widetilde{B}_n(\alpha)=\left(Y_i-(\mathbf x_i-\bar{\mathbf x}_n)^{\top}\widetilde{\bbe}_{nR}\right)_{n:[n\alpha]}$ be two-step averaged $~\alpha$-regression quantile (TARQ) in the model (\ref{1}), with R-estimator $\widetilde{\bbe}_{nR}$ generated by $\varphi_{\lambda}$ in (\ref{8}),
$\lambda\in(0,1)$ fixed. Then, under the conditions \textbf{(A1)}--\textbf{(A3)},
\begin{eqnarray}\label{the}
&(i)~~&n^{1/2}\left[(\widetilde{B}_n(\alpha)-\beta_0-\bar{\mathbf x}_n\bbe)-e_{n:[n\alpha]}\right]=o_p(1)\\
&(ii)~~&n^{1/2}\left|\widetilde{B}_n(\alpha)-\bar{B}_n(\alpha)\right|=o_p(1) \label{the1}
\end{eqnarray} 
as $\ny$,  
uniformly over $\alpha\in(\varepsilon,1-\varepsilon)\subset(0,1), \; \forall \varepsilon\in(0,\frac 12).$
\end{theorem}
\begin{proof} Let us write 
\begin{eqnarray}\label{61}
&&Y_i-(\mathbf x_i-\bar{\mathbf x}_n)^{\top}\widetilde{\bbe}_{nR}\\
&&=e_{ni}+\beta_0+\bar{\mathbf x}_n^{\top}\bbe-(\mathbf x_i-\bar{\mathbf x}_n)^{\top}(\widetilde{\bbe}_{nR}-\bbe), \; i=1,\ldots,n.\nonumber
\end{eqnarray}
We shall study the $[n\alpha]$-quantile of variables $$r_{ni}=e_{ni}-(\mathbf x_i-\bar{\mathbf x}_n)^{\top}(\widetilde{\bbe}_{nR}-\bbe)= Y_i-(\mathbf x_i-\bar{\mathbf x}_n)^{\top}\widetilde{\bbe}_{nR}-\beta_0-\bar{\mathbf x}_n^{\top}\bbe, \; i=1,\ldots,n.$$
Recall the Bahadur representation of sample $\alpha$-quantile $e_{n:[n\alpha]}$ of $e_{n1},\ldots,e_{nn}:$
\begin{eqnarray}\label{Bahadur}
&&n^{1/2}[e_{n:[n\alpha]}-F^{-1}(\alpha)]\nonumber \\
&&=n^{-1/2}[f(F^{-1}(\alpha))]^{-1}\sum_{i=1}^n\{\alpha-\mathcal{I}[e_{ni}<F^{-1}(\alpha)]\}+o(1)
\end{eqnarray}                                                                                                                                         a.s. as $\ny.$                                                                                                                                         Under conditions {\bf (A1)}--{\bf (A3)}, the R-estimator $\widetilde{\bbe}_{nR}$ admits the following asymptotic representation: 
\begin{eqnarray}
\label{12a}
&&n^{\frac 12}(\widetilde{\boldgreek\beta}_{nR}-{\boldgreek\beta})\\
&&=n^{-\frac 12}(f(F^{-1}(\lambda))^{-1}{\mathbf Q}_n^{-1}%
\sum_{i=1}^n(\mathbf x_i-\bar{\mathbf x}_n)\left(\lambda-I[e_{ni}<F^{-1}(\lambda)]\right)+o_p(n^{-1/4}),\nonumber
\end{eqnarray}
hence $\|n^{\frac 12}(\widetilde{\boldgreek\beta}_{nR}-{\boldgreek\beta})\|=\mathcal O_p(1).$ The details for (\ref{Bahadur}) and (\ref{12a}) can be found in \cite{JSP2013}.\\
The $[n\alpha]$ quantile $\tilde{a}_{n}(\alpha)$ of $r_{n1},\ldots,r_{nn}$ is a solution of the minimization
$$\tilde{a}_n(\alpha)=\arg\min_{a\in\R^1}\sum_{i=1}^n\rho_{\alpha}(r_{ni}-a),$$
where $\rho_{\alpha}(z)=|z|\{\alpha \mathcal I[z>0]+(1-\alpha)\mathcal I[z<0]\}, \ z\in\R^1.$
 Denote as  $\psi_{\alpha}$ the right-hand derivative of $\rho_{\alpha},$
i.e. $\psi_{\alpha}(z)=\alpha-\mathcal I[z<0], \; z\in{\R}.$ 
Using Lemma A.2 in \cite{RC1980}, we can show that 
\begin{eqnarray}
\label{7}
&&n^{-\frac 12}\sum_{i=1}^n\psi(r_{ni}-\tilde{a}_n(\alpha))\rightarrow 0, \quad \mbox{ i.e. } \\
&&n^{-\frac 12}\sum_{i=1}^n\left(\alpha-\mathcal I\left[e_{ni}-(\mathbf x_i-\bar{\mathbf x}_n)^{\top}(\widetilde{\bbe}_{nR}-\bbe)<\tilde{a}_n(\alpha)\right]\right)\rightarrow 0 \nonumber
\end{eqnarray}
almost surely as $\ny.$ Notice that $\sum_{i=1}^n(\mathbf x_i-\bar{\mathbf x}_n)={\mathbf 0};$ hence we conclude from \cite{JSP2013}, Lemma 5.5, that it holds 
\begin{eqnarray}
\label{21}
&&
\sup_{\|{\mathbf b}\|\leq C}\Big\{n^{-\frac 12}\Big|\sum_{i=1}^n\Big(\mathcal I[e_{ni}-n^{-\frac 12}{(\mathbf x}_i-\bar{\mathbf x}_n)^{\top}{\mathbf b}<\tilde{a}_n(\alpha)]\nonumber\\
&&\qquad\qquad -\mathcal I[e_{ni}<\tilde{a}_n(\alpha)]\Big)\Big|\Big\}=o_p(1) \; \mbox{  as  } \; \ny,
\end{eqnarray}
 for every $C, \; 0<C<\infty.$ Inserting ${\mathbf b}\mapsto n^{\frac 12}(\widetilde{\boldgreek\beta}_{nR}-{\boldgreek\beta})=
{\mathcal O}_p(1)$ into (\ref{21}), we obtain
\begin{eqnarray}
\label{22}
&& n^{-\frac 12}\Big|\sum_{i=1}^n\Big(\mathcal I\left[e_{ni}-(\mathbf x_i-\bar{\mathbf x}_n)^{\top}(\widetilde{\boldgreek\beta}_{nR}-\bbe)<\tilde{a}_n(\alpha)
\right]\nonumber\\
&&\qquad\qquad -\mathcal I\left[e_{ni}<\tilde{a}_n(\alpha)\right]\Big)\Big|=o_p(1) \; \mbox{ as } \; \ny.
\end{eqnarray}
Combining (\ref{Bahadur}), (\ref{7})--(\ref{22}), we conclude that $n^{1/2}(\tilde{a}_n(\alpha))-e_{n:[n\alpha]})=o_p(1)$ as $\ny,$
hence
\begin{equation}\label{64}
n^{1/2}\left[\widetilde{B}_n(\alpha)-\beta_0-\bar{\mathbf x}_n^{\top}\bbe-e_{n:[n\alpha]}\right]=o_p(1) \; \mbox{ as } \;
\ny
\end{equation}
what gives (\ref{the}). This together with Theorem 2 in \cite{JP2014} implies (\ref{the1}). 
\end{proof}

\begin{remark} $\widetilde{\bbe}_{nR}$ can be replaced by any $\sqrt{n}$-consistent R-estimator of $\bbe.$ However, the score function of type $\varphi_{\lambda}$ is more convenient for computation and hence more convenient for applications. 
Various choices of R-estimators are numerically compared in Section 4.
\end{remark}
\section{Computation and numerical illustrations}
\setcounter{equation}{0}
The simulation study describes the methods for computation of the proposed estimates and illustrates their properties. For the computation of the averaged regression quantile $\bar{B}_n(\alpha),$ the R package quantreg and its function rq() is used; it makes use of a variant of the simplex algorithm. 

Concerning the two-step averaged regression quantile $\widetilde{B}_n(\alpha)$, the most difficult is the first step - the computation of the R-estimator $\widetilde{\bbe}_{nR}$ of the slopes. In the numerical illustration below the score function $\varphi_{\lambda}(u)=\lambda-I[u<\lambda]$ from (\ref{8}) and the approximate scores (i) from (\ref{approx}) are applied. In this case the function rfit() from the R package Rfit could be directly used. For the rfit() function the score function corresponding to (\ref{8}) and (i) of (\ref{approx}) has to be defined, i.e. at point $\frac{\lceil \lambda n \rceil}{ n+1 }$ attaining the value $\lceil \lambda n \rceil -1 - \lambda (n-1) = A_n(i)$. The function rfit() uses the minimization routine optim() which is a quasi-Newton optimizer. This method works well for simple linear regression model but is less precise in case of multiple regression. So, it is better to use the fact that when employing the score function (\ref{8}) with $\lambda = \alpha$ and the approximate scores (i) from (\ref{approx}) the slope components of the regression $\alpha$-quantile and the two-step regression $\alpha$-quantile coincide, $\widehat{\boldgreek\beta}_n(\alpha) = \widetilde{\bbe}_{nR}$ for every fixed $\alpha \in (0,1)$, see \cite{Jur2017}. Therefore, the rq() function from the quantreg package is then used to find the exact solution $\widetilde{\bbe}_{nR}$.    

The averaged regression quantile $\bar{B}_n(\alpha)$ and the two-step averaged regression quantile $\widetilde{B}_n(\alpha)$ (and their inversions) can be used as the estimates of the quantile function (and of the distribution function, respectively) of the model errors.  The behavior of the proposed estimates is illustrated in the following simulation study. 

\begin{figure}[!t]
	\centering
	\includegraphics[width=4.4in]{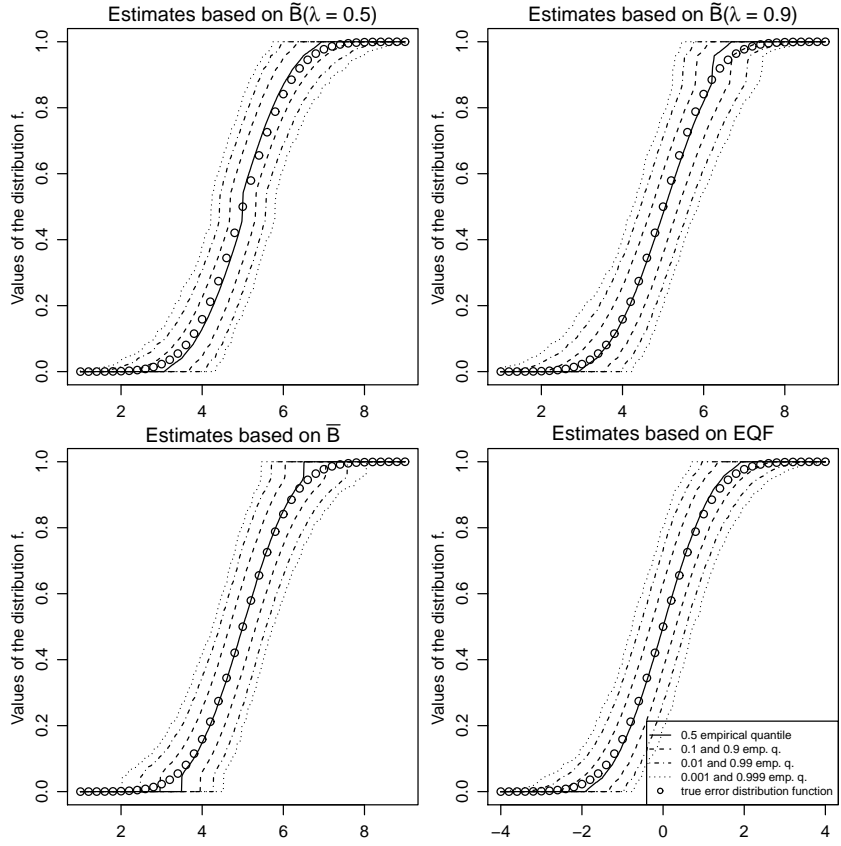}
	\caption{Empirical quantile estimates of normal distribution function based on $\bar{B}$, $\widetilde{B}(0.5)$, $\widetilde{B}(0.9)$ and empirical quantile function (EQF) of errors.}
	\label{fig_norm}
\end{figure}

\begin{figure}[!t]
	\centering
	\includegraphics[width=4.4in]{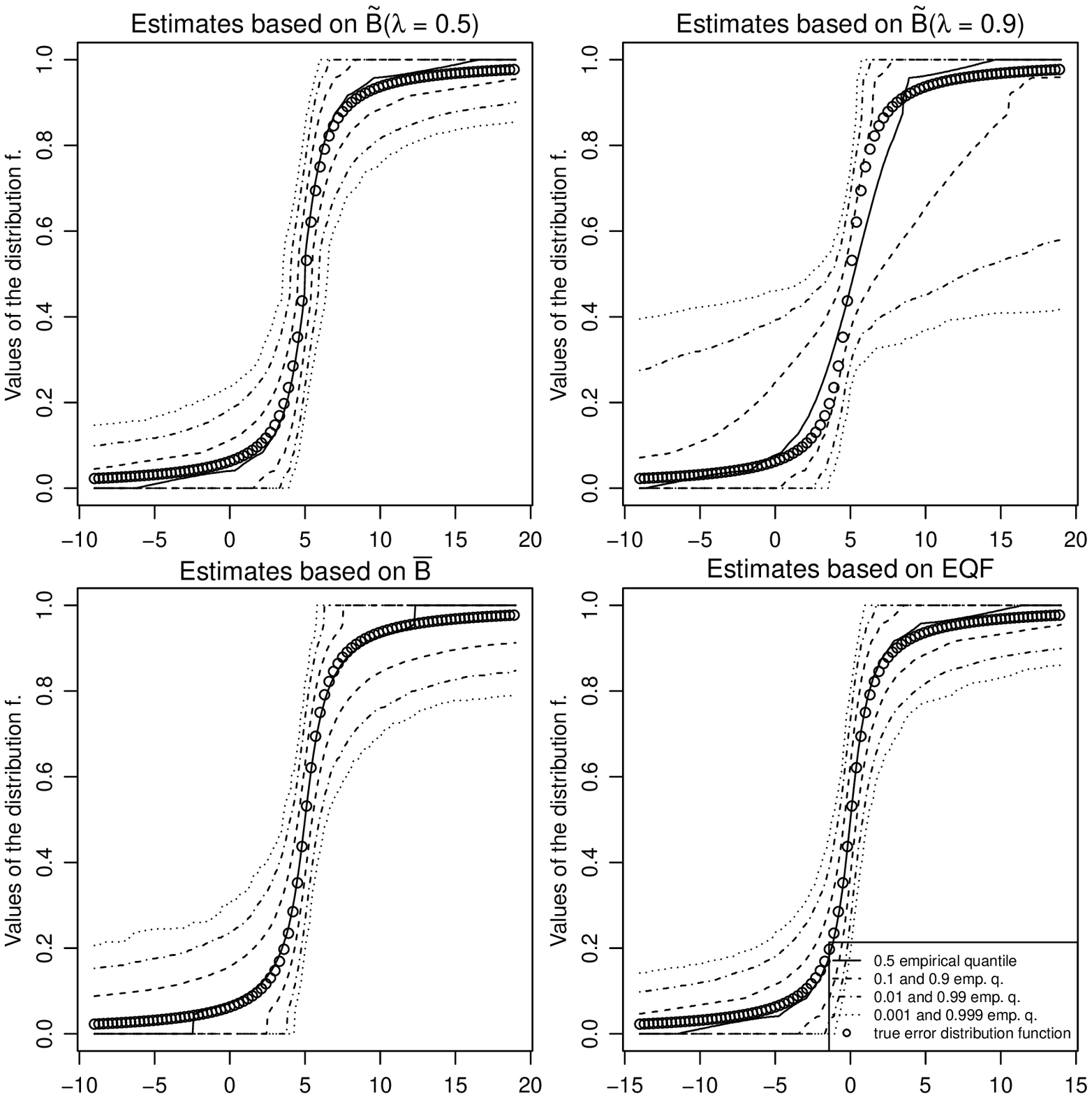}
	\caption{Empirical quantile estimates of cauchy distribution function based on $\bar{B}$, $\widetilde{B}(0.5)$, $\widetilde{B}(0.9)$ and empirical quantile function (EQF) of errors.}
	\label{fig_cauchy}
\end{figure}

\begin{figure}[!t]
	\centering
	\includegraphics[width=4.4in]{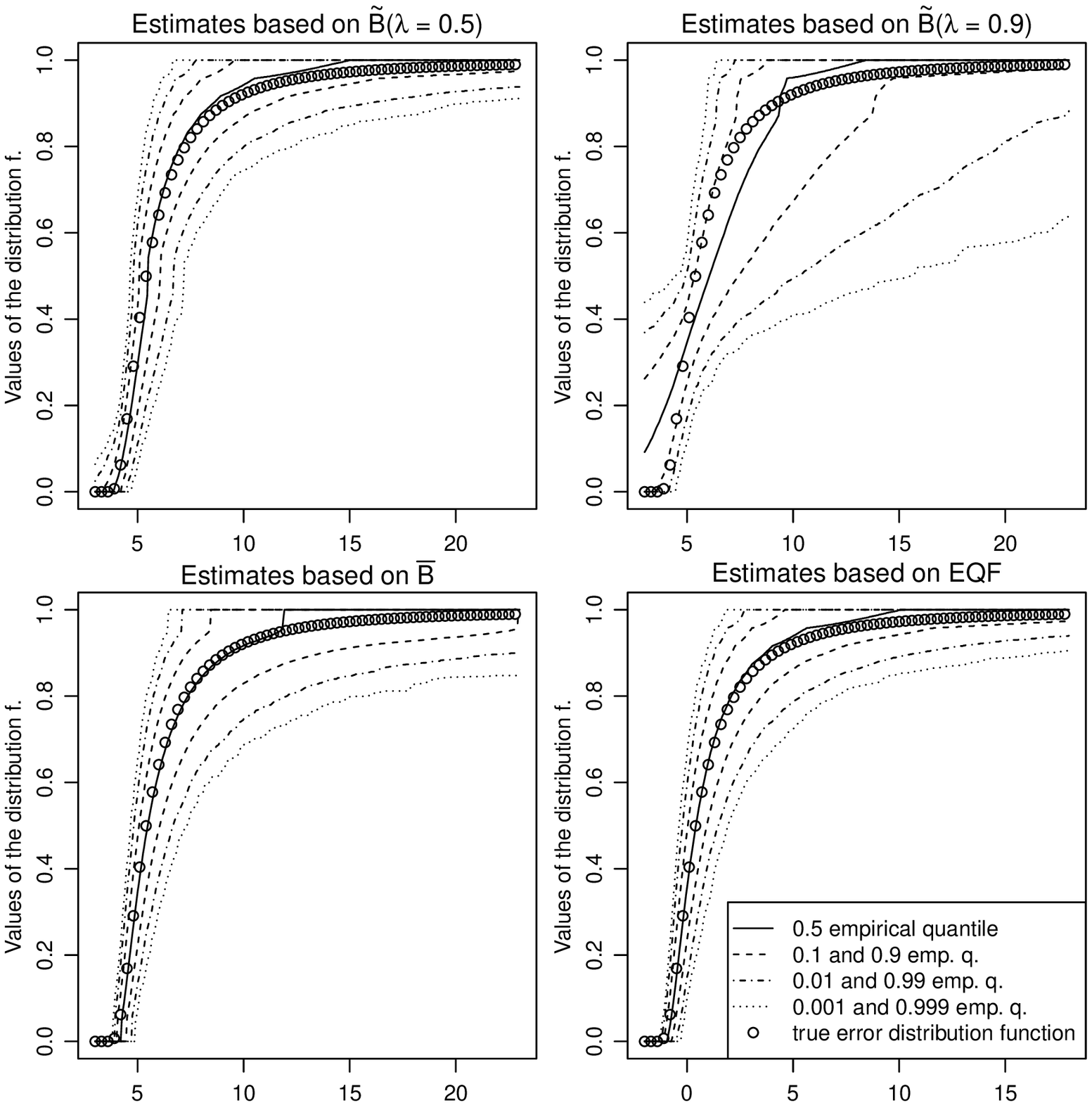}
	\caption{Empirical quantile estimates of GEV ($k = -0.5$) distribution function based on $\bar{B}$, $\widetilde{B}(0.5)$, $\widetilde{B}(0.9)$ and empirical quantile function (EQF) of errors.}
	\label{fig_GEV}
\end{figure}

The regression model 
\begin{equation}
Y_{ni}=\beta_0+{\bf x}_{ni}^{\top}{\boldgreek\beta}+e_{ni}, \quad
i=1,\ldots,n
\end{equation}
is simulated with the following parameters:
\begin{itemize}
	\item sample size $n=25$,
	\item $\beta_0 = 5$,
	\item $\boldgreek\beta = (\beta_1, \beta_2)= (-3, 2)$.
\end{itemize}
The columns of the regression matrix $(x_{11}, \ldots, x_{n1})^{\top}$ and $(x_{12}, \ldots, x_{n2})^{\top}$ are generated as two independent samples from the uniform distributions $U(0,4)$ and $U(-4, 2)$, respectively, and are standardized so that $\sum_{i=1}^{n} x_{ij} = 0, \; j=1,2$.  
The errors $e_{ni}$ are generated from the standard normal, the standard Cauchy or the generalized extreme value (GEV) distribution with the shape parameter $k = -0.5$. For each case $10\,000$ replications of the model were simulated and $\bar{B}_n(\alpha)$ and $\widetilde{B}_n(\alpha)$ and their inversions were computed. For the two-step version $\widetilde{B}_n(\alpha)$ the score-generating function (\ref{8}) with fixed $\lambda = 0.5$ or $0.9$ was used. For a comparison, the empirical quantile function of the errors $e_{ni}$ and its inversion were calculated as well. 
Empirical quantile estimates based on $\bar{B}_n$ and on $\widetilde{B}_n$ 
were then calculated and plotted. Since the figures showing estimates of the true quantile functions and of the true distribution functions look very similar, up to the inversion, only the figures for the distribution functions are presented. The statistical software R was used for all calculations.

The Figures \ref{fig_norm} - \ref{fig_GEV} show the empirical quantile estimates of the normal, Cauchy and GEV distribution functions. The approximation of the distribution functions appears to be very good. We notice that in the case of  two-step regression quantile $\widetilde{B}_n(\alpha)$ with  $\lambda$ fixed, the quality of the estimate is sensitive to the choice of $\lambda$, especially for skewed or heavy-tailed distributions. The choice around $\lambda=0.5$ is generally recommended.

\section{Conclusion}
\setcounter{equation}{0}
The averaged regression quantile $\bar{B}_n(\alpha)$ and its two-step modification $\widetilde{B}_n(\alpha),$ $0<\alpha<1$ appear to be very convenient tools in the analysis of various functionals of the risk  in the situation that the this depends on some exogenous variables, in the intensity that is not fully under our control. The choice of $\alpha\in(0,1)$ provides a balance between the concerns about underestimating and  overestimating the losses in the situation. The increasing $\alpha\nearrow 1$  reflects a greater concern about underestimating the loss, comparing to overestimating. Both $\bar{B}_n(\alpha)$ and $\widetilde{B}_n(\alpha)$ can be advantageously used in estimating the expected shortfall (\ref{shortfall}) and other modern measures of the risk, as well as estimating and testing other characteristics of the market or the everyday practice, based on the functionals of the quantiles.

\end{document}